\documentclass[11pt]{amsart}
\usepackage{amssymb, adjustbox, enumerate, amsbsy, stmaryrd}
\usepackage{geometry}
\usepackage{amsfonts, amssymb, amscd}
\numberwithin{equation}{section}

\usepackage[symbol]{footmisc}

\usepackage{bm}
\usepackage{verbatim}
\usepackage{mathrsfs}
\usepackage{graphicx}
\usepackage{tikz-cd}
\usepackage{subcaption}
\usepackage{listings}
\usepackage{subfiles}
\usepackage[toc,page]{appendix}
\usepackage{mathtools}
\usepackage{comment}
\usepackage{enumerate}
\usepackage{enumitem}
\usepackage[all]{xy}

\usepackage{graphicx}
\graphicspath{{images/}}

\usepackage{appendix}
\usepackage{hyperref}
\hypersetup{
    colorlinks=true,
    citecolor=red,
    linkcolor=blue,
    filecolor=magenta,      
    urlcolor=red,
}
\newcommand{\Spec}{\mathrm{Spec}}

\newcommand{\Cc}{\mathbb{C}}

\newcommand{\Pp}{\mathbb{P}}
\newcommand{\Qq}{\mathbb{Q}}
\newcommand{\Nn}{\mathbb{N}}

\newcommand{\Rr}{\mathbb{R}}

\newcommand{\Zz}{\mathbb{Z}}

\newcommand{\Aut}{\operatorname{Aut}}

\newcommand{\Hom}{\operatorname{Hom}}

\newcommand{\Ker}{\operatorname{Ker}}
\newcommand{\Ima}{\operatorname{Im}}

\newcommand{\Oo}{\mathcal{O}}

\newcommand{\Sing}{\mathrm{Sing}}
\newcommand{\Pic}{\mathrm{Pic}}

\newtheorem{thm}{Theorem}[section]

\newtheorem{cor}[thm]{Corollary}
\newtheorem{lem}[thm]{Lemma}
\newtheorem{prop}[thm]{Proposition}

\theoremstyle{definition}
\newtheorem{defn}[thm]{Definition}
\newtheorem{ques}[thm]{Question}
\theoremstyle{definition}
\newtheorem{rem}[thm]{Remark}

\newtheorem{ex}[thm]{Example}

\newtheorem{theorem}{Theorem}[section]

\theoremstyle{definition}

\newtheorem{example}[theorem]{Example}

\begin{document}

\title{The extension of numerically trivial divisors on a family}
\author{Lingyao Xie}

\address{Department of Mathematics, University of California San Diego, 9500 Gilman Drive, La Jolla, CA 92093}
\email{l6xie@ucsd.edu}

\begin{abstract}
Let $f:X\to S$ be a projective morphism of normal varieties. Assume $U$ is an open subset of $S$ and $L_U$ is a $\Qq$-divisor on $X_U:=X\times_S U$ such that $L_U\equiv_U0$. 
We explore when it is possible to extend $L_U$ to a global $\Qq$-divisor $L$ on $X$ such that $L\equiv_f 0$. In particular, we show that such $L$ always exists after a (weak) semi-stable reduction when $\dim S=1$.

On the other hand, we give an example showing that $L$ may not exist (after any reasonable modification of $f$) if $\dim S\ge 2$, which also gives an $f_U$-nef divisor $M_U$ that cannot extend to an $f$-nef ($\Qq$) divisor $M$ for any compactification of $f|_U$, even after replacing $X_U$ with any higher birational model.
\end{abstract}


\maketitle
\tableofcontents

\section{Introduction}

We work over $\mathbb{C}$ for the sake of simplicity, but it might be true that many of the methods could be (at least partially) applied in positive or mixed characteristics.  Let us start with a simple observation:

\begin{ex}\label{ex: surface ex}
Let $X$ be a smooth surface and $f:X\to S\ni0$ be a projective morphism to a curve $S$ such that 
\begin{enumerate}
    \item $f_*\Oo_X=\Oo_S$,
    \item $f^{-1}(0)=\bigcup_iC_i$ is the irreducible decomposition,
    \item $L$ is a Cartier divisor on $X$ and is numerically trivial over $S\backslash\{0\}$,
\end{enumerate}
then it is always possible to find $a_i\in \Qq$ such that $L+\sum_ia_iC_i\equiv_S0$ by the negative semi-definite properties of the intersection matrices of fibers (cf. \cite[Chapter 7, Lemma 1, Corollary 2]{Fri98}). One just need to notice that 
\[
(L+\sum_ia_iC_i)\cdot f^*(0)=L\cdot f^*(s)=0,~ \forall s\in S, a_i\in\Qq.
\]

Moreover, if we replace (3) with
\begin{itemize}
    \item[$(3)'$] $L$ is numerically effective (nef) over $S\backslash\{0\}$,
\end{itemize}
then respectively we could find $b_i\in \Qq$ such that 
\[
(L+\sum_ib_iC_i)\cdot C_j=d_j\ge0
\]
as long as $\{d_i\}$ satisfies $\sum_ic_id_i=L\cdot f^*(s)\ge0$, where $f^*(0)=\sum_ic_iC_i$. In particular, $L+\sum_ib_iC_i$ is nef over $S$.
\end{ex}

We are wondering if this phenomenon could be generalized to higher dimensions, or we want to know when a relatively numerically trivial (resp. nef) divisor could extend to the global family. The above statement works somehow because we have enough vertical divisors to modify, but it becomes subtle when the dimension of the fibers is $\ge2$: 

The dimension of $N^1(X_s),N_1(X_s)$ (i.e. the Picard rank) could jump for smooth families, while adding vertical divisors does not change intersection numbers (eg. smooth quartic surfaces in $\Pp^3$). Even if the Picard rank is the same, there could be extra curves in the special fiber (eg. degenerations of Hirzebruch surfaces). Therefore it is usually hard to require/verify certain extension of a divisor to be numerically trivial (resp. nef). 

In this paper, we show that the extension of a relatively numerically trivial divisor is always possible after a modification of the family when the base is a curve:

\begin{thm}\label{thm: num div extend for semi-stable fm over curve}
Let $f:X\to S$ be a semi-stable projective family over a smooth curve, and $U\subset S$ is an open subset. Assume $L_U$ is a Cartier divisor on $X_U:=X\times_S U$ such that $L_U\equiv_U0$, then there exists
\begin{itemize}
    \item[(i)] an $m\in\Nn^*$ depending only on $f$, and
    \item[(ii)] a $\Qq$-divisor $L$ on $X$
\end{itemize}
such that
\begin{enumerate}
    \item $L|_{X_U}=L_U$, and
    \item $mL$ is an integral Cartier divisor.
\end{enumerate}
\end{thm}

We remark that $m$ is necessarily as in the surface example (the determinant of the intersection matrix could be very large). We also remark that Theorem \ref{thm: num div extend for semi-stable fm over curve} still holds when the base $S$ is a DVR with perfect residue field, and $f$ is just a locally stable family. 
However, Example \ref{ex: counter example over M_1,2} implies that this is no longer true when $\dim S\ge2$:

\begin{thm}\label{thm: num trivial cannot extend}
There exists a projective contraction $f: X\to S$ and a Cartier divisor $L$ on $X$ such that
\begin{enumerate}
    \item $L|_{X_U}\equiv_U0$ for an open $U\subset S$,
    \item $f$ is semi-stable, in particular $X,S$ are smooth, 
    \item $L'$ is not numerically trivial over $S$ for any $\Qq$-divisor $L'$ on $X$ such that $L'|_{X_U}= L|_{X_U}$.
\end{enumerate}  
\end{thm}

Actually the proof shows that the extension is not possible even after any modifications of the family. As a consequence, it also shows that the extension of a relative nef divisor is not possible in general.

\begin{thm}\label{thm: nef div cannot extend}
There is a projective contraction $X\to S$ of quasi projective normal varieties and $L$ a nef$/S$ divisor but {\bf no} compactification $\bar{X}\to\bar{S}$ and Cartier divisor $L'$ such that 
\begin{enumerate}
    \item $S\subseteq \bar{S}$ is open and $\bar S$ is projective,
    \item $\bar X$ is projective, $\bar{X}\dasharrow X$ is birational
    \item $L'$ is nef over $\bar S$,
    \item $L'|_{\bar{X}_S}\sim_S mL$ for some $m\ge 1$. 
\end{enumerate}
\end{thm}

Nevertheless, it is still unclear if the nefness version remains true when $\dim S=1$:

\begin{ques}
Let $f:X\to S$ be a projective morphism between normal varieties such that $S$ is a smooth curve. Let $U\subset S$ be an open subset and $M$ a Cartier divisor on $X$ such that $M|_{X_U}$ is nef over $U$. Is there a birational model $g:X'\to X$, such that there exist a $\Qq$-divisor $M'$ on $X'$ satisfying 
\begin{enumerate}
    \item $M'|_{X'_U}=(g|_{X'_U})^*M|_{X'_U}$,
    \item $M'$ is nef over $S$.
\end{enumerate}
\end{ques}

In some sense, this is equivalent to the existence of a minimal model (without singularity condition) of $M$ over $S$, which is furthermore equivalent to the existence of a Zariski decomposition of $M$ on higher models, i.e 
$$
M\sim_{\Qq,S} E+P, ~E\ge 0,~P \text{ nef over } S,~N_\sigma(X/S, M)=E,
$$
It seems that if we replace $U$ by the generic point $\eta$, i.e. only ask $M$ to be nef over the generic point, there are counter-examples on family of $\Pp^2$ blowing up ten points (cf. \cite{Les14}) at least for $\Rr$-divisors, since there will be countably many fibers to modify.
\vspace{.5em}

\noindent\textbf{Acknowledgement}.
The authors would like to thank Junpeng Jiao for introducing the question, and thank
Christopher Hacon, Chen Jiang, James McKernan and Ziquan Zhuang for their useful suggestions. The author is supported by a grant from the Simons Foundation.

\section{Preliminaries}

We refer to \cite{KM98} for some basic notations in birational geometry.

\begin{defn}[Numerically trivial divisor]
Let $X$ be a projective variety over a field $k$, a line bundle $L$ is said to be numerically trivial, or $L\equiv0$ if 
\[
L\cdot C=0,~ \forall \text{ complete } C\subset X, ~\dim C=1.
\]
We say $L$ is (relatively) numerically trivial over $S$ (for a projective morphism $f:X\to S$), or $L\equiv_S0$, $L\equiv_f0$, if
$$
L|_{X_s}\equiv0,~\forall~s\in S.
$$
We can define similar notations for $\Qq$-Cartier $\Qq$-divisors.
\end{defn}

\begin{lem}\label{lem: num trivial div descends for rational sing}
Let $g: X'\to X$ be a resolution of rational singularities, $f:X\to S$ is projective and $L'\equiv_S 0$ for some Cartier divisor $L'$ on $X'$. Then $mL'$ descends to a Cartier divisor $L$ on $X$ (i.e, $g^*L=mL'$) for some $m\in\Nn^*$ and $L\equiv_S 0$. Moreover, $m$ only depends on $g$.
\end{lem}

\begin{proof}
It follows from (the proof of) \cite[12.1.4]{KM92}.    
\end{proof}

\begin{lem}\label{lem: num trivial div stable under finite morph}
Let $g: X'\to X$ be a finite morphism between normal varieties$/S$ and $L'$ a Cartier divisor on $X'$, then
\begin{enumerate}
    \item $\deg(g)\cdot g_*L'$ is Cartier, and
    \item $\deg(g)\cdot g_*L'\equiv_S0$ if $L'\equiv_S0$.  
\end{enumerate}
   
\end{lem}
\begin{proof}
(1) follows by locally taking the Norm$(f)$. For (2), we may assume that $g$ is a Galois cover with Galois group $G$, then $\sum_{h\in G}h(L')\equiv_S0$, hence $g^*(g_*L')\equiv_S0$.    
\end{proof}

\begin{lem}\label{lem: num extension is unique for flat fm}
Let $f: X\to S$ be a flat projective morphism such that
\begin{itemize}
    \item $X$ is normal and $S$ is smooth,
    \item $L$ is a Cartier divisor on $X$ and $L\equiv_S 0 $.
\end{itemize}
If $L|_{X_\eta}\sim 0$ for the generic point $\eta\in S$, then $L\sim_{\Qq,S}0$. Furthermore, if $X_s$ is reduced for all the codimension one points $s\in S$, then $L\sim_S0$.     
\end{lem}

\begin{proof}
We may find an open $U\subseteq S$, such that $L|_{X_U}\sim 0$. Hence by adding pullback of divisors on $S$ we can assume that $L$ is a vertical effective divisor (actually we can ask it to be very exceptional), but $-L\equiv_S0$ is nef over $S$, therefore must be a pullback a $\Qq$-divisor $D$ on $S$. If $f^*F$ is reduced for any prime divisor $F$ on $S$, the we can ask $D$ to be integral.    
\end{proof}

\begin{defn}[Simple normal crossing varieties]

Let $W$ be a variety with irreducible components $\{W_i: i\in I\}$. We say that $W$ is a simple normal crossing variety (abbreviated as snc) if the $W_i$ are smooth and every point $p\in W$ has an
open (analytic) neighborhood $p\in U_p\subset W$ and an embedding $U_p\to \Cc^{n+1}$
such that the image of $U_p$ is an open subset of the union of coordinate hyperplanes
$(z_1\cdots z_{n+1}=0)$.

\end{defn}

\begin{defn}[Dual complex]
The combinatorics of a simple normal crossing variety $W$ induce an (ordered) $\Delta$-complex structure $\mathcal{D}(W)$, which can be defined inductively: 
\begin{enumerate}
    \item The vertices (0-simplices) are labeled by the
    irreducible components of $W$, and the maps $\sigma_{W_i}:\Delta^0_{W_i}\to\mathcal{D}_0(W)$ are clear.
    \item Suppose the $r-1$-skeleton $\mathcal{D}_{r-1}(W)$ is defined, then for each 
    \[J=\{j_0,j_1\dots,j_r\}\subset I,~j_0<j_1<\cdots<j_r\] 
    we assign to each irreducible component $Z\subset\bigcap_{l\in J}W_l$ an $r$-simplex $\Delta_Z^r=[v_0,...,v_r]$, and attaching it to $\mathcal{D}_{r-1}(W)$ by identifying the facet $[v_0,\cdots,\hat{v}_i,\cdots,v_r]$ through the map $\sigma_{Z_{j_i}}:\Delta^{r-1}_{Z_{j_i}}\to \mathcal{D}_{r-1}(W)$, where $Z_{j_i}$ is the unique irreducible component in $\bigcap_{l\in J\backslash\{j_i\}}W_{l}$ containing $Z$. Therefore we can define $\mathcal{D}_r(W)$ and the associated maps $\sigma_Z:\Delta^r_Z\to \mathcal{D}_{r}(W)$.   
\end{enumerate}
The $\Delta$-complex $\mathcal{D}(W)$ is called the dual complex of $W$, and the boundary map is given by
\[\partial_r:\Delta_r(\mathcal{D}(W))\to \Delta_{r-1}(\mathcal{D}(W)),~\partial_r(\sigma_{Z})=\sum_{i=1}^r(-1)^{i}\sigma_{Z_{j_i}}, \]
where the $r$-chain group $\Delta_r(\mathcal{D}(W))$ is the abelian group generated by all the $\sigma_Z$ as in $(2)$ above. The simple normal crossing condition guarantees that all the $\sigma_{Z_J}$ is a homeomorphism on to its image, so we can view the simplex $\Delta^{|J|-1}_{Z_J}$ as a subcomplex of $\mathcal{D}(W)$.

\end{defn}

\begin{defn}[Semi-abelian schemes]
A commutative group scheme $G$ which is smooth and of finite type over a field
$k$ is called semi-abelian if its identity component $G^0$ is an extension of an abelian
variety by an affine torus:
\[
\{e\}\to T\to G\to A\to \{e\},
\]
where $T$ is a torus and $A$ is an abelian variety.

Over a general base scheme $S$, an $S$-group scheme $G$ is called semi-abelian if it is smooth
over $S$ and if all its fibers $G_s$ are semi-abelian.    
\end{defn}

\section{Picard schemes}

The key strategy in this paper is to use the Picard scheme. For simplicity, we only work over characteristic 0 in this section. We first recall some basic definitions and results in \cite{Kl05}.

\begin{defn}
For a morphism $f:X\to S$ between schemes and an $S$-scheme $T$, the relative Picard functor $\Pic_{(X/S)}$ is defined by
\[
\Pic_{(X/S)}(T):=\Pic(X_T)/\Pic(T).
\]
Denote its associated sheaves (also as functors) in the Zariski, \'etale, and fppf topologies by
$$
\Pic_{(X/S)\text{(zar)}},~\Pic_{(X/S)\text{(\'etale)}},~\Pic_{(X/S)\text{(fppf)}}~.
$$
If any of the four relative Picard functors above is representable, then the representing scheme is called the Picard scheme and denoted
by $\Pic_{X/S}$, or we say simply that the Picard scheme $\Pic_{X/S}$ exists. 
  
\end{defn}

Even though those functors above could be different, $\Pic_{X/S}$, if exits, is unique by the descent theory. One can see that $\Pic_{X_T/T}$ exists if $\Pic_{X/S}$ does, but it might happen that $\Pic_{X_T/T}$ represents $\Pic_{(X_T/T)}$ but $\Pic_{X/S}$ does not represent $\Pic_{(X/S)}$.

Since we actually need to find a divisor in $\Pic_{(X/S)}(S)$, we would like to know the comparisons between these four functors:

\begin{lem}[{\cite[Theorem 2.5]{Kl05}}]\label{lem: comparison lem}
Assume $\Oo_S\stackrel{\sim}{\longrightarrow}f_*\Oo_X$ holds universally, i.e. $\Oo_T\stackrel{\sim}{\longrightarrow}(f_T)_*\Oo_{X_T}$ for any $S$-scheme $T$.
\begin{enumerate}
    \item[(I).] Then the natural maps (between functors) are injections:
\[
\Pic_{(X/S)}\hookrightarrow\Pic_{(X/S)\text{(zar)}}\hookrightarrow\Pic_{(X/S)\text{(\'etale)}}\hookrightarrow\Pic_{(X/S)\text{(fppf)}}.
\]    
    \item[(II).] All three maps are isomorphisms if also $f$ has a section; the latter two maps
are isomorphisms if also $f$ has a section locally in the Zariski topology; and the last
map is isomorphism if also $f$ has a section locally in the \'etale topology.
\end{enumerate}

\end{lem}

\begin{lem}[cf. {\cite[Section 4]{Kl05}}]\label{lem: Pic exists for nice family}
Let $f:X\to S$ be a flat projective morphism between varieties such that
\begin{enumerate}
    \item every geometric fiber is connected and reduced, and
    \item the irreducible component of any fiber is geometrically irreducible.
\end{enumerate}
Then the Picard scheme $\Pic_{X/S}$ exists and is locally of finite type. Moreover,if fibers are geometrically integral, then $\Pic_{X/S}$ is separated over $S$. 
\end{lem}
\begin{proof}
Condition (1) implies that $\Oo_S\stackrel{\sim}{\longrightarrow}f_*\Oo_X$ holds universally (eg. using cohomology and base change), therefore $\Pic_{X/S}$ exists as an algebraic space locally of finite presentation (cf. \cite[Theorem 4.18.6]{Kl05}). Condition (2) further implies that $\Pic_{X/S}$ is actually a scheme (cf. \cite[Theorem 4.18.1]{Kl05}) locally of finite type. If the fibers are geometrically integral then $\Pic_{X/S}$ is separated by \cite[Theorem 4.8]{Kl05}. 
\end{proof}

In general Condition (2) above is quite subtle even in charateristic 0, unlike Condition (1), it is not enough to check it over all the closed points. For example, consider the following family 
\[
f: X:=\{x^2-ty^2=0\}\subset\Pp^2_{S}\to S:=\Spec~\Cc[t,\frac{1}{t}].
\]
Therefore it is hard to show $\Pic_{X/S}$ exists through Lemma \ref{lem: Pic exists for nice family} if $\dim S\ge 2$. However, it is interesting to mention that $\Pic_{X/S}$ actually exists for the above family (compared to \cite[Example 4.14]{Kl05}). Indeed, if $S'\to S$ is given by 
\[
\Cc[t,1/t]\to \Cc[\bar{t},1/\bar{t}~], t\mapsto \bar{t}^2,
\]
then the base change $f':X'\to S'$ is just two hyperplanes $H^+$, $H^-$ in $\Pp^2_{S'}$ intersecting simple normal crossingly at $[0,0,1]\times S'$, and it is easy to see that
$$
\Pic_{X'/S'}=\bigsqcup_{n,m} S'_{n,m},~(n,m)\in\Zz^2,
$$
where the index $(n,m)$ is given by the degrees on $H^+$ and $H^-$. Gluing $(s',n,m)\sim(-s',m,n)$, we have 
\[
\Pic_{X/S}\cong(\bigsqcup_i S_i)\sqcup(\bigsqcup_{n<m}S'_{n,m})\to S,~S_i\cong S,~ S'_{n,m}\cong S'
\]

The next simple example shows that one should not expect $\Pic_{X/S}$ to behave well even if it exists, it is not separated if the fibers are not integral.

\begin{ex}
Let $X$ be the blow up of $\Pp^1\times C$ at some point over $0\in C$, where $C$ is a smooth curve. Then $\Pic_{X/C}$ is a disjoint union of isomorphic open non-separated subschemes $C_n, n\in\mathbb{Z}$, each $C_n$ is obtained by repeating the origin infinitely many times. More explicitly $C_n=\cup_{a+b=n}C_{a,b}$, where $C_{a,b}$ is the section of $\Pic_{X/C}$ that corresponds to the line bundle with degree $a+b$ on general fibers and degrees $a,b$ respectively on two components of $X_0$.
\end{ex}

It is more practical to work on certain subschemes of the Picard scheme $\Pic_{X/S}$, in particular we study the union $\Pic^0_{X/S}$ of the connected components
of the identity element, $\Pic^0_{X_s/k_s}$, for $s\in S$. By definition $\Pic^0_{X/S}$ is a priori just a set, we will show below that in many interesting cases it can be given the structure of an open subscheme of $\Pic_{X/S}$. We first recall some basic
properties when $S$ is the spectrum of a field.

\begin{lem}[{cf. \cite[Section 5]{Kl05}}]\label{lem: Pic^0 over fields exist}
Let $X$ be a proper variety over a field $k$ of characteristic 0, then $\Pic_{X/k}$ exists and the identity component $\Pic^0_{X/k}$ is quasi-projective and smooth. If $X$ is normal, then $\Pic^0_{X/k}$ is projective, hence an abelian variety.     
\end{lem}

\begin{proof}
$\Pic_{X/k}$ exists and is a disjoint union of open quasi-projective subschemes by \cite[Theorem 4.18.2]{Kl05}. Notice that $\mathrm{char}(k)=0$ then the rest follows from \cite[Proposition 5.3, Theorem 5.4]{Kl05} and the fact that a smooth connected algebraic group over $k$ must be an abelian variety if it is proper.
\end{proof}

\begin{ex}
Let $C\subset\Pp^2_\Cc$ be a cubic curve, then 
\begin{align*}
\Pic^0_{C/\Cc} \cong \begin{cases}
  C  & C \text{ is a smooth elliptic curve,} \\
  \mathbb{G}_a  & C \text{ is a rational curve with a cusp,}  \\
  \mathbb{G}_m    &  C \text{ is a nodal cubic.}
\end{cases}
\end{align*}
\end{ex}

In the above example, a Cartier divisor is numerically trivial if and only iff it belongs to $\Pic^0(X)$, but this may not be true in general, for instance when there are torsions in $H^2(X,\Zz)$. Nevertheless, this will become true after a multiplication.

\begin{lem}\label{lem: num 0 multiple to Pic 0}
Let $X$ be a projective variety over a field $k$ with $\mathrm{char}(k)=0$, there is an $m\in \Nn^*$, such that if $L\equiv 0$, then $mL\in\Pic^0_{X/k}$.
\end{lem}

\begin{proof}
It suffices to show this after base change to $\bar{k}$, then this follows from \cite[Theorem 6.3, Propostion 6.12]{Kl05}. 
\end{proof}

The following proposition almost immediately follows from \cite[Proposition 5.20]{Kl05}, but we provide a full proof here for the readers' convenience. 

\begin{prop}\label{prop: Pic^0 of good family exists}
Let $f:X\to S$ be a flat projective contraction and any fiber is of semi-lc singularity type. Assume that $\Pic_{X/S}$ exists and is locally of finite type, then $\Pic_{X/S}$ has an open group subscheme $\Pic^0_{X/S}$ of finite type whose fibers are the $\Pic^0_{X_s/k_s}$. Furthermore, we have
\begin{enumerate}
    \item $\Pic^0_{X/S}$ is separated.
    \item If $S$ is reduced, then $\Pic^0_{X/S}$ is smooth over $S$, and (additionally)
    \item if $X_s$ is normal for any $s\in S$, then $\Pic^0_{X/S}$ is proper over $S$ and is closed in $\Pic_{X/S}$. 
\end{enumerate}
\end{prop}
\begin{proof}
By the assumptions we know $\Pic_{X_s/k_s}$ equals to the base change $\Pic_{X/S}\times \Spec(k_s)$ and there is a section $\sigma_0$ (the identity) corresponding to the trivial divisor such that $\Pic^0_{X_s/k_s}\subseteq \Pic_{X_s/k_s}$ is the connected component containing $\sigma_0(s)$. Let $\Pic^0_{X/S}$ denotes the union of $\Pic^0_{X_s/k_s}$ (as set) in $\Pic_{X/S}$.

Since all the fibers are potentially slc, we have that $h^i(X_s,\Oo_{X_s})$ is locally constant on $S$ by \cite[Corollary 2.64]{Kol23}. In particular, $\dim\Pic^0_{X_s/k_s}=h^1(X_s,\Oo_{X_s})$ is locally constant. Since we are in char 0, $\Pic^0(X_s/k_s)$ is smooth (geometrically reduced and irreducible), hence by \cite[15.6.3]{EGA IV$_3$} $\Pic_{X/S}\to S$ is universal open at any points of $\Pic^0_{X/S}$, then by \cite[15.6.4]{EGA IV$_3$} $\Pic^0_{X/S}$ is an open subset of $\Pic_{X/S}$.

For (1), one just need to consider the case when $S$ is a DVR. Let $L$ be a numerical trivial Cartier divisor such that $L|_{X_\eta}\sim 0$, then one should show that $L\sim_S 0$. 

We may assume that $L$ is effective and supported on the closed fiber, then the nefness ensures that it must be the pullback of some multiple of the closed point (as a Cartier divisor). 

For (2), since $\Pic^0_{X/S}\to S$ is universal open, then $\Pic^0_{X/S}\to S$ is flat after base change to $S_{red}$ by \cite[15.6.7]{EGA IV$_3$}, hence smooth since the fibers are smooth.

For (3), $\Pic^0_{X_s/k_s}$ are proper by Lemma \ref{lem: Pic^0 over fields exist} since $X_s$ are normal. Now that $\Pic^0_{X/S}\to S$ is separated and $\Pic^0(X_s/k_s)$ is geometrically connected, $\Pic^0_{X/S}$ is proper by \cite[15.7.11]{EGA IV$_3$}. Then $\Pic^0_{X/S}$ is closed in $\Pic_{X/S}$ since $\Pic_{X/S}$ is separated.
\end{proof}

\begin{lem}\label{lem: section in proper Pic0 extends}
Assume $S$ has klt type singularities and $\Pic^0_{X/S}$ is proper over $S$, then any generic section $\sigma_\eta:\Spec(k_\eta)\to \Pic_{X_\eta/k_\eta}$ uniquely extends to a section $\sigma: S\to \Pic^0_{X/S}$, i.e. $\sigma|_\eta=\sigma_\eta$.   
\end{lem}

\begin{proof}
Since $\Pic^0_{X/S}$ is proper over $S$, we can take the closure $S'$ of $\sigma_\eta$ in $\Pic^0_{X/S}$ and we get a proper birational morphism $\phi: S'\to S$. 

In order to get the section $\sigma$,  we must prove that $\phi:S'\to S$ is an isomorphism, and by Zariski's main theorem it suffices to show that $\phi$ is finite since $S$ is normal. Indeed, for any closed point $s\in S$, $\phi^{-1}(s)$ is rationally chain connected by \cite[Theorem 1.2]{HM07}, hence contains a rational curve if $\dim \phi^{-1}(s)\ge1$. However, this implies that $\Pic^0_{X_s/k_s}$ contains a rational curve, which is impossible since a smooth proper connected algebraic group over $\Cc$ must be a projective abelian variety.   

\end{proof}

We emphasize that a section in $\Pic_{X/S}$ does not necessarily correspond to an actual divisor on $X$ since the functor $\Pic_{(X/S)}$ might not be representable. In fact, we know $\Pic_{X/S}$ would represent $\Pic_{(X/S)\text{(\'etale)}}$ if the fibers are reduced (cf. Lemma \ref{lem: comparison lem}), as in this case we have local \'etale sections, but for a section in $\Pic^0_{(X/S)\text{(\'etale)}}(S)$, the induced 1-cocyle which takes values in $\Pic(U_{ij})$ for the \'etale cover $\bigsqcup_i U_i\to S$ may not be exact, which becomes the obstruction of gluing together to a line bundle on $S$. Therefore usually a generically finite base change is necessary, after which we do have a global section of $f$ and hence $\Pic_{(X/S)}$ is representable by $\Pic_{X/S}$.

\begin{cor}\label{cor: num div extends when Pic0 proper}
Let $f:X\to S$ be a projective locally stable family (cf. \cite{Kol23}) of normal varieties over a reduced scheme $S$. Let $L_\eta$ be a numerically trivial line bundle on $X_\eta$, where $\eta\in S$ is the generic point. 
\begin{enumerate}
    \item Then possibly after a generically finite base change, there exists a $\Qq$-Cartier $\Qq$-divisor $L'$ on $X$ such that $L\equiv_S 0$ and $L_\eta= L'|_{X_\eta}$. 
    \item If $L|_{X_\eta}=L_\eta$ for some line bundle on $X$, then $L\equiv_S0$. (No base change needed.)
\end{enumerate}

\end{cor}

\begin{proof}
By taking the Stein factorization we may assume that $f_*\Oo_X=\Oo_S$.
The assumptions imply that $X$ is generically normal and we can select a closed integral subvariety $S''\subset X$ (eg. obtained by taking general hyperplane intersections) such that $S'\to S$ is generically finite and projective. Then let $S'\to S''$ be a resolution of singularity, and by base change to $S'$, we may assume that $f$ have a section and $S$ is smooth. 

Since $f$ is locally stable with normal fibers, Lemma \ref{lem: Pic exists for nice family} implies that $\Pic_{X/S}$ exists and is of finite type, and by Lemma \ref{lem: comparison lem} $\Pic_{X/S}$ actually represent the functor $\Pic_{(X/S)}$. Since the fibers of a locally stable family are slc type, by Proposition \ref{prop: Pic^0 of good family exists} we know $\Pic^0_{X/S}$ is an open subscheme of $\Pic_{X/S}$ and is proper over $S$.

By Lemma \ref{lem: num 0 multiple to Pic 0} we may assume that $L_\eta\in \Pic^0_{X_\eta/k_\eta}$. Now $L_\eta$ corresponds to a section 
\[
\sigma_\eta: \Spec(k_\eta)\to \Pic^0_{X_\eta/k_\eta}=(\Pic^0_{X/S})_\eta,
\]
and by Lemma \ref{lem: section in proper Pic0 extends} $\sigma_\eta$ extends to a section $\sigma$ in $\Pic^0_{X/S}(S)$, which by the representability gives a line bundle $L'$ on $X$ such that $L'|_{X_\eta}\sim L_\eta$, then after rechoose $L'$ we can ask $L'|_{X_\eta}=L_\eta$, and of course $L'\equiv_S 0$ because $L'|_{X_s}$ corresponds to the point $\sigma(s)\in \Pic^0_{X_s/k_s}$ for any $s\in S$, which is algebraically equivalent to the trivial divisor by definition, hence also numerically trivial. 

For (2), line bundles $L_\eta$ and $L$ give sections
$$
\sigma_{L_\eta}: \Spec (k_\eta)\to \Pic_{X_\eta/k_\eta}, ~\sigma_L: S\to \Pic_{X/S}
$$ such that $(\sigma_L)_\eta=\sigma_{L_\eta}$. By Lemma \ref{lem: num 0 multiple to Pic 0} there is an $m\in \Nn^*$ such that $mL_\eta\in \Pic^0(X_\eta)$, hence we have the corresponding section $\sigma_{mL_\eta}$ and $\sigma_{mL_\eta}(\eta)\in\Pic^0_{X_\eta/k_\eta}$. Notice that $\Pic_{X/S}$ is separated since all the fibers are geometrically integral, and $\Pic^0_{X/S}$ is closed in $\Pic^0_{X/S}$ by Proposition \ref{prop: Pic^0 of good family exists} (3), the section $\sigma_{mL}$ must be the unique extension of $\sigma_{mL_\eta}$ in $\Pic^0_{X/S}$. And the $L\equiv_S0$ follows since it can be checked over the closed points.

\end{proof}

In order to study $\Pic^0_{X/S}$ more explicitly, we need to know $\Pic^0_{X_s/k_s}$ for each fibers. We focus on the case when $X_s$ is reduced and simple normal crossing, since this might be the best we can hope for after a modification of a family.

\begin{prop}\label{prop: pic of snc is semi-abelian}
Let $W=\bigcup_{i\in I}W_i$ be a projective simple normal crossing variety over $\Cc$, then $\Pic^0_{W/\Cc}$ is semi-abelian. Moreover, the torus part is given by the character of 
$$H_1(\mathcal{D}(W),\Zz)/H_1(\mathcal{D}(W),\Zz)_{tor},$$ 
or equivalently, the identity component of $H^1(\mathcal{D}(W),\mathbb{G}_m)$ (as an algebraic group).
\end{prop}

\begin{proof}
Let $\mu:W^n=\bigsqcup_iW_i\to W$ be the normalization, then the pullbacks of line bundles induce a morphism between algebraic groups:
\[\mu^*:\Pic^0_{W/\Cc}\to\Pic^0_{W^n/\Cc}=\prod\Pic^0_{W_i/\Cc},\]
and there exists algebraic subgroups $\Ker\mu^*\subset \Pic^0_{W/\Cc},~\Ima \mu^*\subset\Pic^0_{W^n/\Cc}$ such that 
\[\{e\}\to \Ker\mu^*\to\Pic^0_{W/\Cc}\to\Ima \mu^*\to\{e\} \]
is exact. The fact that $\Ima\mu^*$ is abelian is clear since $\Pic^0_{W^n/\Cc}$ is an abelian variety, so we only need to show that $\Ker\mu^*$ is a torus. We can just look at the set of closed points $\Ker\mu^*(\Cc)$ since $\Cc$ is algebraically closed.

Let $[L]\in \Ker\mu^*(\Cc)$ be a closed point, given by a line bundle $L$ on $W$, then by definition $L$ is trivial on each $W_i$. We claim that 
\begin{align}\label{eq: Pic is H^1}
S:=\{[L]\in\Pic(W)~|~\mu^*L\simeq\Oo_{W^n}\}\cong H^1(\mathcal{D}(W),\Cc^*).    
\end{align}
Note that $\Ker\mu^*(\Cc)\subset S\subset\Pic_{W/\Cc}(\Cc)$ as groups, it is easy to see that $\Ker\mu^*$ is the identity component of $H^1(\mathcal{D}(W),\mathbb{G}_m)$
as the largest torus. Then by the universal coefficient theorem, we have 
\[
H^1(\mathcal{D}(W),\mathbb{G}_m)\cong\Hom(H_1(\mathcal{D}(W),\Zz),\mathbb{G}_m),
\]
hence the statement follows.

It remains to prove (\ref{eq: Pic is H^1}): Considering $L$ as a total space gluing together from $\mu^*L$, the gluing information is totally determined by how $L|_{W_i}$ and $L|_{W_j}$ glue together over any irreducible $Z_{ij}\subset W_i\cap W_j,i<j$. We first choose isomorphisms $\psi_i:L|_{W_i}\simeq\Oo_{W_i}$, then $\psi_i|_{Z_{ij}}\circ(\psi_j|_{Z_{ij}})^{-1}$ gives an element in $\Aut_{Z_{ij}}(\Oo_{Z_{ij}})\simeq\Cc^*$. Therefore $L$ gives rise to an 1-cochain $$
\psi\in\Hom(\Delta_1(\mathcal{D}(W)),\Cc^*)$$ (depending on the choices) and we are done if we can show the following
\begin{itemize}
    \item[(i)] The automorphism condition (the cochains given by those $L\simeq\Oo_W$) is equivalent to the cochain being exact.
    \item[(ii)] The compatibility condition (the cochains given by some $L$ such that $[L]\in S$ ) is equivalent to the cochain being closed.
\end{itemize}

First notice that if we only change $\psi_i$ to $\psi'_i$, and fix other $\psi_j$, then $\psi_i'\circ\psi_i^{-1}$ gives an element $\alpha_i\in\Cc^*$ and we can check that the induced cochains differ by a $\partial(\beta_i)$, where $\beta_i$ is the 0-cochain such that
\[
<\beta_i,\Delta^0_{W_j}>=(\alpha_i)^{\delta_{ij}}
\]
(notice that $<\partial(\beta_i),\Delta^1_{Z_{lj}}>=<\beta_i,\Delta^0_{W_l}-\Delta^0_{W_j}>=\alpha_i^{\delta_{il}-\delta_{ij}},~ l<j$.)

Now if $\psi:L\simeq\Oo_W$, which induces isomorphisms $\{\psi_i\}_{i\in I}$ by restrictions, then the induced cochain $\phi$ is the trivial element, hence if we have different choices of $\{\psi'_i\}_{i\in I}$, which induces a cochain $\phi'$, then we see $$
\phi'=\frac{\phi'}{\phi}=\prod_i\partial(\beta_i)
$$  
where $\{\beta_i\}_{i\in I}$ is constructed as above by considering $\psi_i'\circ\psi_i^{-1}$ for all $i\in I$. In particular, $\phi'$ is exact. We have proved (i).

For (ii),  let $\phi$ be a cochain in $\Hom(\Delta_1(\mathcal{D}(W)),\Cc^*)$ and $<\phi,\Delta^1_{Z_{ij}}>=\alpha_{ij}$. Then $\phi$ is induced by a line bundle $L$ such that $[L]\in S$ if and only if the space $W^n\times\Cc$ glues to the total space $L$ via the relations generated by
\begin{align}\label{eq: relations on Z_ij}
(x_i,\lambda)\sim(x_j,\alpha_{ij}\lambda),~x_i\in W_i,x_j\in W_j,~\mu(x_i)=\mu(x_j),~i<j  
\end{align}
and is further equivalent to for any $i\in I$,
\[(x_i,\lambda)\sim(x_i,\lambda'),~x_i\in W_i\iff \lambda=\lambda'.\]
We see that it is necessary to require
\[
1=\alpha_{ij}\alpha_{jk}\alpha_{ik}^{-1}=<\phi,\Delta^1_{W_i}-\Delta^1_{W_j}+\Delta^1_{W_k}>=<\partial(\phi),\Delta^2_{Z_{ijk}}>
\]
for any irreducible $Z_{ijk}\subset W_i\cap W_j\cap W_k,~i<j<k$. Therefore $\phi$ needs to be closed 1-cochains. We claim that this is also sufficient: 

We can check for each stratum in $W$, let $Z_J\subset\bigcap_{i\in J}W_i$ be any irreducible components and suppose 
$$
x\in Z_J, x\notin W_i ,\forall i\notin J,
$$ 
then $\phi|_{\Delta^{|J|-1}_{Z_J}}$ (by restriction to the subcomplex $\Delta^{|J|-1}_{Z_J}$) is also closed, hence exact since $\Delta^{|J|-1}_{Z_J}$ is simplicial. If we only consider locally around $x$, then the associated dual complex is just $\Delta^{|J|-1}_{Z_J}$ and by (i) there is a line bundle such that its induced co-chain is exactly $\phi|_{\Delta^{|J|-1}_{Z_J}}$. Therefore $W^n\times\Cc$ glues to a line bundle via the relations (\ref{eq: relations on Z_ij}) at least
locally around $x$, and hence glues globally since $x\in W$ can be random.

\end{proof}

\begin{rem}
We make a few comments before we moving on:
\begin{enumerate}
    \item Assuming we have a family $f:X\to S$ such that $X_s$ is simple normal crossing for all $s\in S$ and the dual complex $\mathcal{D}(X_s)$ is simply connected, Proposition \ref{prop: pic of snc is semi-abelian} implies that the $\Pic^0_{X/S}$ (if exists) should be proper over $S$. Even though $\Pic_{X/S}$ might not exists because of Condition (2) of Lemma \ref{lem: Pic exists for nice family}, $\Pic_{X/S}$ is an algebraic space of finite presentation and the author believes that the similar argument in Proposition \ref{prop: Pic^0 of good family exists} should imply that $\Pic^0_{X/S}$ is at least a proper algebraic space over $S$, whose fibers are abelian varieties. Then one can show by the same argument in Corollary \ref{cor: num div extends when Pic0 proper} that any numerically trivial divisor extends.  
    \item Proposition \ref{prop: pic of snc is semi-abelian} should holds in a much more general setting, at least when $W$ is semi-log canonical. Instead of only considering the combinatorics of $W$ itself, one needs to take all the normalization maps (of strata) into considerations. When $W$ is simple normal crossing, it is enough to see all the maps from the dual complex, but even for normal crossing varieties, self intersections make things subtle. The maps would impose more conditions on the gluing data but one should still be able to show that the $\Pic^0(W)$ is semi-abelian.
\end{enumerate}

\end{rem}

\section{N\'eron models}

Coming back to our question, we see from above that the remained obstruction to extending a numerically trivial divisor is the failure of properness of $\Pic^0_{X_s/k_s}$ on special fibers. One can easily see that the naive idea to directly extend sections failed in the following examples:

\begin{ex}\label{ex: degeneration to Ga}
Let $E$ be an elliptic curve and $C$ a smooth curve, $\pi:E\times C \to C$, blow up a point $p$ over $0\in C$ and then contract the elliptic fiber $E_0$ to get $f:X\to C$, then let $L=C_2-C_1$, where $C_1$ is a section passing through $p$ and $C_2$ is another general section of $\pi$, then the birational transform $L_X$ of $L$ on $X$ is a non-$\Qq$ Cartier divisor.

On the Picard schemes level, $f$ is a flat degeneration of elliptic curves to a rational curve with a cusp. By abusing notations we still use $C_i$ as their birational transforms on $X$, then with the help of $C_1$, we have an canonical isomorphism
\[
(\Pic^0_{X/S},0_S)\simeq (X-\{q\},C_1), ~q=\Sing(X),
\]
and $L$ corresponds to the section $C_2$, therefore $L$ can extend if and only if $C_2$ does not pass through $q\in X$. Moreover, $mL$ can extend if and only if $m(C_2-C_1)|_{E_0}\sim0$.
\end{ex}

Situation gets a little bit better for a good degeneration of family: 

\begin{ex}\label{ex: surface example on M_1,2}
Let $f:X\to S$ be a family of stable curves of genus one with two marked points:
\begin{enumerate}
    \item $S$ is a smooth curve and $f$ is smooth over $U:=S-\{0\}$, i.e. $X_s$ is an elliptic curve for any $s\in S-\{0\}$.
    \item $X_0=f^{-1}(0)=C_1\cup C_2$, where $C_1,C_2\simeq\mathbb{P}^1$ intersecting transversally at two points,
    \item sections $S_i,i=1,2$ intersects the curve $C_i$ and at the smooth locus of $X_0$.
\end{enumerate}
More concretely, there are cases that $X$ is smooth, then one can check that we must have 
$$
C_1\cdot C_1=C_2\cdot C_2=-2=-C_1\cdot C_2
$$

Let $g:X\to Y$ be the contraction of $C_2$ with $g(C_2)=p\in Y$, 
Then the section $S_{1,Y}:=g_*S_1$ induces a canonical isomorphism 
\[
\phi: (\Pic^0_{Y/S},0_S)\simeq (Y-\{p\},S_{1,Y})
\]
and we also have isomorphism $g^*:\Pic^0_{Y/S}\simeq\Pic^0_{X/S}$. Notice that $S_{2,Y}=g_*S_2$ is a section of $Y\to S$ passing through $p$, so $S_{2,Y}|_{Y_U}$, viewed as a section $\sigma_U$ of $\Pic^0_{X_U/U}\to U$, could not extend to a global section of $\Pic^0_{X/S}\to S$. Equivalently, the divisor $(S_2-S_1)|_{X_U}$ could not extend to a Cartier divisor that is numerically trivial over $S$.

However, as in the Example \ref{ex: surface ex}, we see $2(S_2-S_1)+C_2\equiv_S 0$. That means $2\cdot\sigma_U\in\Pic^0_{X_U/U}(U)$ extends to a global section in $\Pic^0_{X/S}(S)$, even though $\sigma_U$ can not. 

We have the multiplication morphism of group schemes:
\[
\alpha:\Pic^0_{X/S}\stackrel{[2]}{\longrightarrow}\Pic^0_{X/S},
\]
which gives a rational map $\hat\alpha: Y\dashrightarrow Y$ and one can check that we can resolve it by the blow up $g:X\to Y\ni p$:
\[
\hat\alpha\circ g=\bar\alpha: X\to Y,
\]
which is a finite morphism of degree 4.
let $\{q_1,q_2\}=C_1\cap C_2$, then there is an $S$-group scheme structure on $X-\{q_1,q_2\}$, whose central fiber $X_0$ is two copy of $\Cc^*$ and we have:
$$
\beta: \Pic^0_{X/S}\hookrightarrow G:=X-\{q_1,q_2\}, ~\beta\circ\bar\alpha=\alpha
$$
Now $S_2$ can be viewed as a section of $G$, and maps to $\bar\alpha(S_2)$ in $\Pic^0_{X/S}$, which is the section corresponding to the Cartier divisor $2(S_2-S_1)+C_2$.

\end{ex}

\vspace{.5em}

We are wondering if the above phenomenon holds in general, i.e, a rational/generic section of $\Pic^0_{X/S}$ could extend to a global section after a multiplication. As we see above, this would follows if we show that there is a larger group scheme $G$ of finite type, such that 
\begin{enumerate}
    \item there is an open immersion $\Pic^0_{X/S}\hookrightarrow G$,
    \item $G$ satisfies the extension property: any generic section extends to a global section.
\end{enumerate}

The group schemes $\Pic^0_{X/S}$ can be regarded as a degeneration of abelian varieties, and one can see from Example \ref{ex: degeneration to Ga} that the degeneration should be ``good" for the extension argument to work. Fortunately for us, by Proposition \ref{prop: pic of snc is semi-abelian}, we have a nice correspondence between
\begin{align*}
\text{semi-stable reduction of } f:X\to S 
\end{align*}
and 
\[
\text{semi-abelian reduction of } \Pic^0_{X/S}.
\]

When $S$ is a DVR or a Dedekind scheme, there is a very nice theory about the degeneration of abelian varieties, and the group scheme $G$ has a name, called the ``N\'eron model" of $(\Pic^0_{X/S})_\eta$.

Actually \cite{Kol25} implies that the model $G$ could be construct explicitly in the following way (just as in Example \ref{ex: surface example on M_1,2}): 

One take a random compactification $Y\to S$ of $\Pic^0_{X/S}\to S$, and take a resolution of singularity $Y'\to Y\to S$ such that 
$$
Y_s \text{ is simple normal crossing }\forall s\in S,~Y_\eta=(\Pic^0_{X/S})_\eta
$$
Then we can run a $K_{Y'}$-MMP over $S$, which will terminate with a minimal model
$\mu:Y_{min}\to S$ with terminal singularities, and it turns out that the $\mu$-smooth locus
\[
G:=\{y\in Y_{min}~|~\mu \textit{ is smooth at } y\}
\]
is exactly the ``N\'eron model" of $(\Pic^0_{X/S})_\eta$. However, we will not go that far since what we need here is the existence, and then the rest would follow from Proposition \ref{prop: semi-abelian open immerse to Neron model}.

\vspace{1em}

Now we start to recall some standard facts about N\'eron models from \cite{BLR90}.

\begin{defn}[{\cite[1.2.1, 1.2.6]{BLR90}}]
Let $S$ be a Dedekind scheme such that $\eta\in S$ is the generic point. Let $\tilde{X}_\eta$ be a smooth and separated $k(\eta)$-scheme of finite type. 

An $S$-scheme $X$ such that $X_\eta\cong\tilde{X}_\eta$ is defined to be a N\'eron model of $X_\eta$ if $X$ is a smooth separated $S$-scheme of finite type, which has the following universal property, called N\'eron mapping property:

{\it For each smooth $S$-scheme $Y$ and each $k(\eta)$-morphism $u_\eta: Y_\eta\to X_\eta$,
there is a unique $S$-morphism $u: Y\to X$ extending $u_\eta$.}

\noindent In particular, taking $Y=S$, this means that every generic section extends to a global section.

If $\tilde{X}_\eta$ is a group scheme over $k(\eta)$, then the N\'eron model $X$ is necessarily an $S$-group scheme which extends the group scheme structure on $X_\eta\cong \tilde{X}_\eta$. As a typical example, one can easily check that (cf. \cite[1.2.8]{BLR90}) an abelian scheme over $S$ is the N\'eron model of its generic fiber.
 
\end{defn}

A highly non-trivial fact is that the N\'eron model always exists for abelian varieties:

\begin{thm}[{\cite[1.4.3]{BLR90}}]
Let $S$ be a connected Dedekind scheme with field of fractions $k_{\eta}$ and let
$A_\eta$ be an abelian variety over $k(\eta)$. Then $A_{k(\eta)}$ admits a global N\'eron model $A$ over $S$.    
\end{thm}

Now we are ready to prove our main theorem:

\begin{proof}[Proof of Theorem \ref{thm: num div extend for semi-stable fm over curve}]
We prove with a weaker assumption that $f$ is locally stable (rather than semi-stable) with klt generic fiber, then one can check ($S_2$ and $R_1$) that $X$ is normal, and normal after any finite base change $S'\to S$ if $S'$ is normal. Actually, by \cite[Proposition 2.15]{Kol23} $X,X'$ have only klt, hence also rational singularities. Additionally, we can take the Stein factorization and assume that $f_*\Oo_X=\Oo_S$. 

We first show that we can reduce to case when $X_s$ is simple normal crossing for all $s\in S$, and $f$ admits a section.

As in the proof of Corollary \ref{cor: num div extends when Pic0 proper}, we can easily find a projective morphism $S'\to S$, where $S'$ is smooth such that the base change $f':X'\to S'$ admit a section, notice that this is a finite base change and by Lemma \ref{lem: num trivial div stable under finite morph} it suffice to show $L_U$ could extends to a relatively numerically trivial divisor on $X'$. Hence we may assume $f$ admits a section and then any further base change would also has a section.

By \cite[Chapter II]{KKMS73} (Semi-stable reduction) we know that there is a modification of $f$:
\begin{center}$\xymatrix{
X''\ar@{->}[dr]_{g}\ar@{->}[r]^{h}& X'\ar@{->}[r]^{\varphi_X}\ar@{->}[d]^{f'} & X\ar@{->}[d]^{f}\\
 & S'\ar@{->}[r]^{\varphi} & S\\
}$
\end{center}
where $f':X'=X\times_{S}S'\to S'$ is the base change of the finite morphism $\varphi:S'\to S$ and $h$ is projective birational, such that 
\begin{enumerate}
    \item $S'$ and $X''$ are smooth, and
    \item any fiber $X''_{s'}$ is reduced (as scheme) with simple normal crossing support.
\end{enumerate}
Since $X'$ has only rational singularities, by Lemma \ref{lem: num trivial div stable under finite morph} again and Lemma \ref{lem: num trivial div descends for rational sing} it suffices to show $L_U$ extends to a relatively numerically trivial divisor on $X''$ (over $S'$), hence we may also assume that $f$ is semi-stable with simple normal crossing fibers.

Now we can check that
\begin{enumerate}
    \item[(i)] $\Pic_{X/S}$ exists and represents $\Pic_{(X/S)}$.
    \item[(ii)] $\Pic^0_{X/S}$ exists and is an open subscheme of $\Pic_{X/S}$.
    \item[(iii)] $\Pic^0_{X/S}$ is semi-abelian, and the generic fiber is an abelian variety.
\end{enumerate}
For (i), since $S$ is a curve and the generic fiber $X_\eta$ is normal and geometrically connected, we only need to check the two conditions of Lemma \ref{lem: Pic exists for nice family} for closed fibers, which is obvious.

\noindent For (ii), fibers of locally stable family are slc so it follows from Proposition \ref{prop: Pic^0 of good family exists}.

\noindent For (iii), $\Pic^0_{X_\eta/k_\eta}$ is abelian since $X_\eta$ is normal, $\Pic^0_{X_s/k_s}$ is semi-abelian by Proposition \ref{prop: pic of snc is semi-abelian}.

By Lemma \ref{lem: num 0 multiple to Pic 0}, possibly after a multiplication, we may also assume that $(L_U)_\eta\in\Pic^0(X_\eta)$. Lemma \ref{lem: num extension is unique for flat fm} tells us it suffices to extend $L_U|_{X_\eta}$, since then the extension will automatically agree with $L_U$ over $U$. Let $\sigma_\eta\in\Pic^0_{X_\eta/k_\eta}$ be the generic section corresponding to the line bundle $(L_U)|_{X_\eta}$, we want to show that $\sigma_\eta$ extends to a global section in $\Pic^0_{X/S}$ after a multiplication.

Let $A$ be the N\'eron model of $(\Pic^0_{X/S})_\eta$ over $S$, then there is a canonical morphism induced by the universal N\'eron mapping property:
\[
\mu:\Pic^0_{X/S}\to A,
\]
where $\mu$ is an open immersion of group schemes, and is an isomorphism on the identity components by Proposition \ref{prop: semi-abelian open immerse to Neron model} below. 

The N\'eron mapping property says that $\sigma_\eta$ extends to a section $\sigma$ in $G$. Since $\Pic^0_{X_\eta/k_\eta}$ is abelian and $S$ is a curve, there exists a finite set 
\[
\Gamma:=\{s\in S~|~A_s \text{ is not abelian}\}
\]
The numbers of connected components of $G_s,s\in \Gamma$ is finite since $A_s$ is of finite type, and we can take $m$ to be the lowest common multiple of them. Since 
$$
\sigma\in A(S):=\Hom_S(S,A)
$$
is a element of an abelian group, $m\sigma$ makes sense and can be viewed as a section of $A\to S$.
Notice that 
$A_s/\Pic^0_{X_s/k_s}$ is a finite group whose order divides $m$, therefore we have
\[
(m\sigma)(s)=m\cdot\sigma(s)\in\Pic^0_{X_s/k_s},~\forall s\in S,
\]
which means $m\sigma$ actually factors through $\Pic^0_{X/S}$. Therefore $m\sigma$ corresponds to a line bundle $L'$ on $X$ such that $L'|_{X_\eta}=mL_U|_{X_\eta}$ since $m\sigma|_\eta=m\sigma_\eta$ by our construction.

\end{proof}

\begin{prop}[{\cite[7.4.3]{BLR90}}]\label{prop: semi-abelian open immerse to Neron model}
Let $A_\eta$ be an abelian variety over $k(\eta)$ with N\'eron model $A$, and let $G$ be a semi-abelian $S$-scheme such that $G_\eta\simeq A_\eta$. Then the canonical morphism $G\to A$ is an open immersion (of group schemes) and it is an isomorphism
between the identity components of $G_s$ and $A_s$ for $s\in S$.     
\end{prop}

We don't provide the proof here but we remark that the essential point is the torsion points in a semi-abelian variety are topologically (Zariski, \'etale, fppf) dense.

\section{Higher dimensional counterexamples}

When the base has $\dim S\ge2$, we shall not expect a relatively numerically trivial divisor could extend in general, even after any reasonable modification of the family.

\begin{example}\label{ex: counter example over M_1,2}
Let $\overline{\mathcal{M}}_{1,2}$ be the moduli stack of stable genus one curves with two marked points (on the smooth locus) and $(\mathcal{U}, \Sigma_1,\Sigma_2)\to\overline{\mathcal{M}}_{1,2}$ be the universal family. There are five types of closed points on $\overline{\mathcal{M}}_{1,2}$, whose fibers on the universal family are
\begin{enumerate}
    \item[(I)] an elliptic curve with two distinct points,
    \item[(II)] a nodal cubic with two distinct points,
    \item[(III)] a circle consisting of two $\mathbb{P}^1$ with two points on different components,
    \item[(IV)] an elliptic curve attached with a $\Pp^1$, two distinct points on the $\mathbb{P}^1$,
    \item[(V)] a nodal cubic attached with a $\mathbb{P}^1$ on the smooth locus, two distinct points on the $\mathbb{P}^1$.
\end{enumerate}
Let $\mathcal{L}$ be the line bundle/Cartier divisor on $\mathcal{U}$ defined by $\Sigma_1-\Sigma_2$. Then the restriction of $\mathcal{L}$ over the points of type (I,II,IV,V) would be numerically trivial, but not over the point of type (III). 

Now let $S$ be any integral projective variety with a morphism $S\to\overline{\mathcal{M}}_{1,2}$ whose image contains the locus of type (II), then there exists an open subset $U\subseteq S$ such that the pullback line bundle/Cartier divisor $\mathcal{L}_U$ is numerically trivial over $U$, but we claim that 

\textit{For any $m\neq0$, there is no Cartier divisor $L$ on the base change family $\mathcal{U}_S$ such that $L\equiv_S 0$ and $L|_U\sim_U m\mathcal{L}_U$.}

Indeed, assume on the contrary that such $(m,L)$ exists, possibly after base change we may assume $S$ is smooth and $\mathcal{U}_S$ is normal. By Lemma \ref{lem: num extension is unique for flat fm} we must have $L|_{\mathcal{U}_V}\sim_{V} m\mathcal{L}_V$, where the open subset $V\subset S$ is the locus of type (I,II,IV,V).

Let $\overline{\mathrm{M}}_{1,2}$ be the coarse moduli space, which is actually isomorphic to a weighted $(2,3)$-blow up of the weighted projective plane at
\[
[1:0:0]\in \Pp(1,2,3).
\]
Let $C'$ be the closure of locus of type (II) in $\overline{\mathrm{M}}_{1,2}$, set theoretically it is the union of points of type (II,III,V) and we can check that $C'$ is an irreducible curve (actually isomorphic to $\Pp^1$). Now let 
$$\phi:S\to\overline{\mathcal{M}}_{1,2}\to\overline{\mathrm{M}}_{1,2}$$
be the composition, then $C'\subseteq\phi(S)$ from our assumption and we may choose a projective irreducible curve $C\subset S$ dominating $C'$. Notice that $V\cap C$ is dense and open in $C$ and by further base change we may assume $C$ is smooth.

Let $(\mathcal{U}_C,\Sigma_{1,C},\Sigma_{2,C})\to C$ be the base change family, then we have 
\[
m(\Sigma_{1,C}-\Sigma_{2,C})=m\mathcal{L}_C,~(m\mathcal{L}_C)|_{\mathcal{U}_{V\cap C}}\sim_{V\cap C} L|_{\mathcal{U}_C},
\]
and by construction the closed fibers of the family are also of type (II,III,V). By Lemma \ref{lem: Pic exists for nice family} we know $\Pic_{\mathcal{U}_C/C}$ exists, then we can compute the $\Pic^0_{\mathcal{U}_C/C}$ by using Proposition \ref{prop: Pic^0 of good family exists}. Fiberwise computation tells us that 
$$
(\Pic^0_{\mathcal{U}_C/C})_t=\Pic^0_{\mathcal{U}_t/k_t}\simeq \mathbb{G}_m, ~\forall t\in C,
$$
therefore the group scheme is actually a $\mathbb{G}_m$ bundle with a canonical identity section (given by the trivial divisor), and notice that 
$$
\Aut(\mathbb{G}_m,\{1\})\simeq\mathbb{Z}/2\mathbb{Z},
$$ which is generated by $\lambda\mapsto\frac{1}{\lambda}$, so possibly after a cyclic \'etale cover $C''\to C$ to kill the monodromy we may assume that $\Pic^0_{\mathcal{U}_C/C}\simeq\mathbb{G}_m\times C$ . 

Now the Cartier divisor $m(\Sigma_{1,C}-\Sigma_{2,C})$ is numerically trivial over $V\cap C$, it corresponds to a rational section
$$
\alpha_{m,C}:C\dashrightarrow (\Pic^0_{\mathcal{U}_C/C},0_C)\simeq(\mathbb{G}_m\times C,\{1\}\times C).
$$
where $\alpha_{m,C}$ is defined over points of $V\cap C$, which is of type (II,V). We have
\begin{itemize}
    \item For any $t\in C$ of type (II), $\alpha_{m,C}(t)=\Sigma_{1,t}-\Sigma_{2,t}\neq 0$ in $\Pic^0(\mathcal{U}_t)$.
    \item For any $t\in C$ of type (V), $\alpha_{m,T}(t)=0\in\Pic^0(\mathcal{U}_t)$.
\end{itemize}
Since $L|_{\mathcal{U}_C}\equiv_C 0$, it gives a section $\alpha_L$ of $\Pic^0_{\mathcal{U}_C/C}$ such that
$$
(\alpha_{m,C})|_{V\cap C}=(\alpha_L)|_{V\cap C},
$$
then we must have $\alpha_{m,C}=\alpha_L$ is a regular section since $\Pic^0_{\mathcal{U}_C/C}$ is separated. Then we the composition
\[
C\stackrel{\alpha_{m,C}}{\longrightarrow}\mathbb{G}_m\times C\stackrel{\mathrm{pr_1}}{\longrightarrow} \mathbb{G}_m
\]
is non-constant. However, this is impossible since $C$ is a proper variety but $\mathbb{G}_m$ is affine.

\end{example}

\begin{proof}[Proof of Theorem \ref{thm: num trivial cannot extend}]
We still use the notations in Example \ref{ex: counter example over M_1,2}. Let $S\to\overline{\mathcal{M}}_{1,2}$ be any morphism from a smooth projective variety $S$ such that the induced morphism $S\to\overline{\mathrm{M}}_{1,2}$ is dominant. Let $g:X\to \mathcal{U}_S$ be any projective birational morphism, where $X$ is smooth. Let $U\subset S$ be any non-empty open subset, then we claim that $(mg^*\mathcal{L})|_{X_U}$ cannot extends to a relatively numerically trivial Cartier divisor $L$ for any $m\neq0$.

Indeed, if such $L$ exists, notice that $\mathcal{U}_S$ is normal and klt, then $L$ descends to a relatively numerically trivial divisor on $\mathcal{U}_S$ by Lemma \ref{lem: num trivial div descends for rational sing}, which contradicts the argument in Example \ref{ex: counter example over M_1,2}.

Then the theorem follows by finding any semi-stable reduction of the family $\mathcal{U}_S\to S$.

\end{proof}

\begin{proof}[Proof of Theorem \ref{thm: nef div cannot extend}]
We will use the notations in Example \ref{ex: counter example over M_1,2}. Let $S'\to \overline{\mathcal{M}}_{1,2}$ be a dominant morphism from a proper smooth variety and the open locus $S\subset S'$ of type (I). Then it suffices to verify the statement for every flat $\bar X\to \bar S$ such that
\begin{enumerate}
    \item[(i)] $\bar X\to \mathcal{U}_{\bar S}, \bar S\to S'$ are proper birational morphisms,
    \item[(ii)] $\bar X$ is normal and $\bar S$ is smooth.
\end{enumerate}

By Example \ref{ex: counter example over M_1,2}, we know that the divisors $m(\Sigma_{1,S}-\Sigma_{2,S})$ cannot extend to a relatively numerically trivial divisor on $\bar X$ over $\bar{S}$ for any $m\neq 0$. Now both
$$
\Sigma_{1,S}-\Sigma_{2,S},~\Sigma_{2,S}-\Sigma_{1,S}
$$ 
are relative nef divisors, if both of them could extends to a relatively nef divisor $L_{+},L_{-}$ on $\bar X$ satisfying (3)-(4), $\bar X$ has klt sing. Then $L_{+}+L_{-}$ is a nef Cartier divisor such that the restriction over $S$ is trivial, therefore we may think $L_{+}+L_{-}$ as an effective vertical Cartier divisor, hence has to be a pullback from $\bar S$ since it is nef over $\bar S$. This implies that $L_{+}+L_{-}$ is actually numerically trivial over $\bar S$, so the relatively nef divisors $L_{+},L_{-}$ must also be relatively numerically trivial. This is impossible by Example \ref{ex: counter example over M_1,2} (cf. the proof of Theorem \ref{thm: num trivial cannot extend}).

\end{proof}


\end{document}